\documentclass[12pt]{amsart}
\usepackage{amscd,amssymb,graphics,lineno}
\usepackage{mathrsfs} 

\newtheorem{theorem}{Theorem}[section]
\newtheorem{corollary}[theorem]{Corollary} 
\newtheorem{lemma}[theorem]{Lemma}

\theoremstyle{definition}

\theoremstyle{remark}

\numberwithin{equation}{section}

\def\N{{\mathbb N}}

\def\Iso{{\mathrm{Iso}}}

\def\e{\varepsilon}
\def\Homeo{{\mbox{\rm Homeo}\,}}

\newcounter{quest}
\stepcounter{quest}

\def\sym#1{{[\hspace{-.28em}(#1)\hspace{-.28em}]}}

\def\d{\delta}
\def\stm{\setminus}
\def\o{\omega}
\def\sB{\mathcal B}
\def\sF{\mathcal F}
\def\sN{\mathcal N}
\def\sbs{\subset}

\def\obr{^{-1}}

\begin{document}

  \title
{Projectively universal countable metrizable groups}

\author[Vladimir G. Pestov]{Vladimir G. Pestov$^{1}$}

\address{V.G.P.: Department of Mathematics and Statistics, University of
Ottawa, Ottawa, ON, K1N 6N5, Canada}

\address{Departamento de Matem\'atica, Universidade Federal de Santa Catarina, Trindade, Florian\'opolis, SC, 88.040-900, Brazil}

\email{vpest283@uottawa.ca}

\author{Vladimir V. Uspenskij}

\address{V.V.U.: Department of Mathematics, 321 Morton Hall, Ohio
University, Athens, Ohio 45701, USA}

\email{uspenski@ohio.edu}

\thanks{$^1$ Special Visiting Researcher of the program Science Without Borders of CAPES (Brazil), processo 085/2012.}


\date{October 26, 2015}


 \begin{abstract} 
We prove that there exists a countable metrizable topological group $G$
such that every countable metrizable group is isomorphic to a quotient of $G$.
The completion $H$ of $G$ is a Polish group such that every Polish group is isomorphic
to a quotient of $H$. 
 \end{abstract}

\maketitle

\section{Introduction}
A topological group is {\em Polish} if it is homeomorphic to a complete separable metric space. It has been known for about 30 years
that there exist {\em injectively universal} Polish groups \cite{Usp86, Usp90}, that is, Polish groups $G$ such that
every Polish group $H$ is isomorphic (as a topological group) to a (necessarily closed) subgroup of $G$. Those examples in particular answered a Scottish Book question by Schreier-Ulam (question 103 in \cite{SB}); for a recent new example, see \cite{Yaacov}.
A few years ago L.~Ding proved \cite{Ding},
answering a long-standing question of A.~Kechris (see \cite[Problem 2.10]{Kechris}, \cite[Problem 1.4.2]{BK}), 
that there also exists  a {\em projectively universal}, 
or {\em couniversal}, 
Polish group, that is, such a Polish group $G$ that every Polish group $H$ is isomorphic (as a topological group) to the quotient
group $G/N$ for some closed invariant subgroup $N\triangleleft G$. 

The aim of this note is to provide a shorter proof of a stronger
theorem: there exists a projectively universal countable metrizable group. The completion of such a group is a projectively
universal Polish group (Theorem~\ref{thm1}), so our result indeed implies that of Ding. We give two constructions in 
sections~\ref{s:p} and~\ref{s:u}, due to the first and the second author, respectively. 

We mention that  a projectively universal Abelian Polish group was constructed in \cite{PSW}, and an injectively universal Abelian Polish group was constructed in \cite{Shk}. The question remains open, due to Kechris,
whether every Polish group is isomorphic to a quotient of a subgroup of the unitary group $U(\ell^2)$; the answer is positive in the Abelian case \cite{GP}, \cite{Usp08}.

We thank the anonymous referee for remarks that have led to an improved presentation. 

\section{Polish groups as completions of countable groups}
\label{s:Polish}

We are going to explain why the completion of a projectively universal countable metrizable group is a projectively
universal Polish group. All our groups are Hausdorff. Recall that for every topological group $G$ its {\em Rajkov completion}
$\hat G$ is defined as the completion with respect to the upper bound of the left and right uniformities, see \cite{RD, ArhTk}.
If $G$ is a topological subgroup of the group $\Iso(X)$ of linear isometries of a Banach space $X$ (every topological group
admits such an embedding), then we can take for $\hat G$ the closure of $G$ in $\Iso(X)$. If $G$ is a (necessarily metrizable)
group with a countable base, then $G$ is Polish if and only if $G$ is Rajkov complete, that is, $G=\hat G$. Indeed, if $G$ is
Rajkov complete, then the two-sided uniformity on $G$ is complete and metrizable, hence admits a compatible complete metric.
It follows that $G$ is Polish. Conversely, suppose that $G$ is Polish. Every Polish space $X$ is a $G_\d$ subset in any 
Hausdorff space $Y\supset X$ containing $X$ as a dense subset. 
If $G\ne\hat G$, then $G$ is a dense $G_\d$ subset in the Polish group $\hat G$, and so is the translate $Gx$ for any 
$x\in \hat G\stm G$. Since $G\cap Gx=\emptyset$, we obtain a contradiction with the Baire Category Theorem. 

Thus every Polish group can be viewed as the Rajkov completion of any of its dense countable subgroups. 
Recall that the quotient of any Polish group by a closed invariant subgroup is Polish \cite[Proposition 1.2.3]{BK}, \cite[Theorem 4.3.26]{ArhTk}.
Theorem 4.3.26 in \cite{ArhTk} actually is more general and deals with \v{C}ech-complete groups, but for spaces 
with a countable base `Polish' and `\v{C}ech-complete' are equivalent. 
 
  \begin{theorem}
\label{thm1}
    Let $G$ be a projectively universal countable metrizable group. Then
    the completion $\hat{G}$ is a projectively universal Polish group. 
  \end{theorem}

\begin{proof}
Write the given Polish group as the completion $\hat H$ of a countable metrizable group $H$. 
Write $H$ as a quotient of $G$. Then $H=G/N$ can be identified with a dense subgroup
of $\hat G/\hat N$ \cite[Ch. 3, \S 2, Proposition 21]{Bour}. The quotient $\hat G/\hat N$ of a Polish
group $\hat G$ is Polish, hence Rajkov complete. It follows that $\hat H$ can be identified with  $\hat G/\hat N$, 
a quotient of $\hat G$.
 \end{proof} 

\section{The first construction}
\label{s:p}

\subsection{}
We will construct a couniversal countable metrizable group using a technique developed by Roelcke and Dierolf \cite{RD}.

For a sequence $(B_n)_{n=1}^{\infty}$ of subsets of a group $G$, define their symmetric product as follows:
\[\sym{ B_n}=\sym{ B_n}_{n=1}^{\infty} = \bigcup_{n=1}^{\infty}\bigcup_{\sigma\in S_n} B_{\sigma(1)}\cdot B_{\sigma(2)}\cdot \ldots \cdot B_{\sigma(n)}.\]
For example, the proof of the Birkhoff--Kakutani theorem implies:
\begin{lemma}
  If $(V_n)_{n=1}^{\infty}$ be a sequence of neighbourhoods of the identity in a topological group $G$ such that $V_n^{-1}=V_n$ and $V_{n+1}^2\subseteq V_n$ for all $n$. Then the  for every $k$ one has
\[
\sym{V_n}_{n=k+2}^{\infty}\subseteq V_k.\]
\label{l:bk}
\end{lemma}

  Let now $\mathcal F$ be a filter of subsets of a group $G$.
  For a mapping $\Phi\colon G\to {\mathcal F}$ denote
  \[{\mathcal V}_{\Phi} =\bigcup_{g\in G} g^{-1}(\Phi(g)\cup\Phi(g)^{-1})g.\]
  According to \cite{RD}, the sets of the form $\sym{{\mathcal V}_{\Phi_n}}_{n=1}^{\infty}$, where $(\Phi_n)$ runs over all sequences of maps from $G$ to $\mathcal F$, 
  form a neighbourhood basis at identity in the finest group topology on $G$ in which ${\mathcal F}\to e$. E.g., in \cite{P} this was used to describe a neighbourhood basis in the free topological group on a uniform space.
  
  Since it is obvious that every set of the form $\sym{{\mathcal V}_{\Phi_n}}$ contains an element of the filter $\mathcal F$, the result of Roelcke and Dierolf follows from Lemma \ref{l:bk}, as well as the following easily verifiable results.
  
  \begin{lemma}
    Every set $\sym{{\mathcal V}_{\Phi_n}}$ is symmetric.
  \end{lemma}
  
  \begin{lemma} Assuming the sequence of mappings $\Phi_n\colon G\to {\mathcal F}$ is pointwise monotone (that is, $\Phi_{n+1}(g)\subseteq \Phi_n(g)$ for each $g\in G$), we have 
    \[\sym{{\mathcal V}_{\Phi_{2n}}}^2\subseteq\sym{{\mathcal V}_{\Phi_n}}.\]
  \end{lemma}
 
  \begin{lemma} For each $h\in G$, denote $\Phi^h$ the right translate of $\Phi$ by $h$, that is, $\Phi^h(g)=\Phi(gh)$. Then
    \[h^{-1}\sym{{\mathcal V}_{\Phi^h_n}}h\subseteq \sym{{\mathcal V}_{\Phi_n}}.\]
  \end{lemma}
  
  \begin{corollary}
          Let $\Phi_n\colon G\to {\mathcal F}$ be a pointwise monotone sequence of mappings. Then the sets
  \[\sym{{\mathcal V}_{\Phi^h_{kn}}},~~h\in G,~~k\in\N,\]
  form a subbasis at identity for a group topology on $G$.
  \label{c:basis}
\end{corollary}
  
\subsection{}
  On the set $\N$ of natural numbers select a decreasing sequence of sets $(U_n)$ with empty intersection and such that $U_0=\N$ and $U_n\setminus U_{n+1}$ is infinite for each $n$. 
  
  Let $F(\N)$ denote the free group on the set $\N$ of free generators. 
  Define a sequence of functions $\phi_n\colon F(\N)\to\mathcal \N$ by setting $\phi_n(k^{\pm 1}) = {n+k}$ for all $k\in\N$ and extending each $\phi_n$ over the free group by recursion on the length of a reduced word $w=n_1^{\e_1}\ldots n_k^{\e_k}$, $\e_i=\pm 1$, $n_i\in \N$ as follows:
  \[\phi_n(w) \equiv  \phi_n(n_1^{\e_1}\ldots n_k^{\e_k}) =
    \max\left\{\phi_{\phi_n(n_1^{\e_1}\ldots n_{k-1}^{\e_{k-1}})}(n_k),
  \phi_{\phi_n(n_2^{\e_2}\ldots n_k^{\e_k})}(n_1)\right\}.\]
  Now set $\Phi_n(g)=U_{\phi_n(g)}$. This is easily seen to be a pointwise monotone family of maps from the free group to a filter generated by $(U_n)$. 
  
  As in Corollary \ref{c:basis}, the family $(\Phi_n)$ defines a metrizable group topology on $F(\N)$. This topology is Hausdorff: for every $k$, the neighbourhood $\sym{{\mathcal V}_{\Phi_{kn}}}$ is contained in the normal subgroup generated by $U_k$, and such subgroups separate points in the free group.
  
  We will now show that when equipped with the above topology, $F(\N)$ is a couniversal countable metrizable group.

\subsection{}
Let $(V_n)$ be a countable basis of neighbourhoods of identity of a metrizable group $G$. For every element $g$, define the {\em scale} of $g$ with regard to the fixed basis as a function $\theta_g\colon \N\to \N$, as follows:
\[\theta_g(n) =\min\{m\colon g^{-1}V_mg\cup gV_mg^{-1}\subseteq V_n\}.\]
For a general Polish group such as $\Homeo_+[0,1]$ for instance, it is easy to see that the scales formed with regard to any neighbourhood basis form a cofinal subset of $\N^\N$. Not so for countable metrizable groups.

\begin{lemma}
        A countable metrizable group $G$ admits a symmetric neighbourhood basis $(V_n)$ of identity satisfying $V_{n+1}^2\subseteq V_n$ for all $n$ and with regard to which the scale of every element satisfies $\theta_g(n)\leq n+m$ for a suitable $m=m(g)$.
  \label{l:scale}
\end{lemma}

\begin{proof}
        Enumerate $G=\{g_m\colon m\in\N_+\}$ and choose a basis $V_n$ recursively so that for each $n$, $g_m^{-1}V_{n+1}g_m\cup g_mV_{n+1}g_m^{-1}\subseteq V_n$ whenever $m\leq n$.
\end{proof}

\begin{lemma}
        Let $U\subseteq\N$ be infinite, let $V$ be a countably infinite set, and let $m\colon V\to \N$ be a function whose image contains $0$. There exists a surjection $f\colon U\to V$ such that 
        \[\forall k\in U,~~m(f(k))\leq k.\]
        \label{l:sur}
\end{lemma}

\begin{proof} Identify $V$ with $\N$ in such a way that $m(0)=0$. Define $f(k)$ recursively in $k$ as the smallest element of $\N=V$ not yet chosen provided its image under $m$ does not exceed $k$, and $0$ otherwise. 
\end{proof}

Now choose a basis $(V_n)_{n=0}^\infty$ for $G$ as in Lemma \ref{l:scale}, where $V_0=G$. The corresponding function $m$ on $G$ satisfies $m(e)=0$. For every $n\in\N$ apply Lemma \ref{l:sur} to the sets $U=U_n\setminus U_{n+1}$, $V=V_n$, and the function $m$. We obtain a surjection $f_n\colon U_n\setminus U_{n+1}\to V_n$ with $m(f_n(k))\leq k$ for all $k\in U_n\setminus U_{n+1}\subseteq\N$. Amalgamate all $f_n$ to obtain a surjection $f\colon\N\to G$. For all $k,n\in\N$ one has 
\begin{equation}
        \label{eq:ineq}
        \phi_n(k) = n+k \geq n+m(f(k)) \geq \theta_{f(k)}(n).\end{equation}
Extend $f$ to a surjective group homomorphism $\bar f\colon F(\N)\to G$. 

\begin{lemma} For all $x\in F(\N)$ and $n\in\N$, one has $\theta_{\bar f(x)}(n)\leq \phi_n(x)$.
  \end{lemma}
  
  \begin{proof}
          Induction on the reduced length, $\ell(x)$, of $x$. For $\ell(x)=1$, this is Eq. (\ref{eq:ineq}) and the symmetry of $\theta$. Suppose the result holds for $\ell(x)\leq \ell$. Let $x=x_1^{\e_1}\ldots x_{\ell+1}^{\e_{\ell+1}}$, where $\e_i=\pm 1$, be an irreducible word. We have:
    \begin{eqnarray*}
            f(x)^{-1} V_{\phi_n(x)}f(x) 
      &\subseteq & 
      \bar f(x_2^{\e_2}\ldots x_{\ell+1}^{\e_{\ell+1}})^{-1} V_{{\phi_n(x_2^{\e_2}\ldots x_{\ell+1}^{\e_{\ell+1}})}}\bar f(x_2^{\e_2}\ldots x_{\ell+1}^{\e_{\ell+1}})\\
      \mbox{\tiny induction hypothesis} &\subseteq & 
      \bar f(x_2^{\e_2}\ldots x_{\ell+1}^{\e_{\ell+1}})^{-1} V_{{\theta_{\bar f(x_2^{\e_2}\ldots x_{\ell+1}^{\e_{\ell+1}})}}(n)}\bar f(x_2^{\e_2}\ldots x_{\ell+1}^{\e_{\ell+1}})\\
      &\subseteq & V_n.
      \end{eqnarray*} 
Similarly, $f(x) V_{\phi_n(x)}f(x)^{-1}\subseteq V_n$, and we conclude.
  \end{proof}

  As an application of the lemma, $\bar f\left( {\mathcal V}_{\Phi_{n}}\right)\subseteq V_n$ for every $n\in\N$. Now
Lemma \ref{l:bk} implies that for every $k$,
$\bar f\left(\sym{{\mathcal V}_{\Phi_{n}}}_{n=k+2}^{\infty}\right)\subseteq V_k$.
  Consequently, the homomorphism $\bar f$ is continuous. 
  Since each neighbourhood of identity in $F(\N)$ contains one of the sets $U_n$ which is being mapped by $\bar f$ onto $V_n$, the homomorphism $\bar f$ is also open.

\section{The second construction}
\label{s:u}

Another construction of 
a projectively universal countable metrizable group is based on the following idea:
take a collection of size $2^\o$ representing, up to an isomorphism, all possible
countable metrizable groups. The product of this collection contains a countable dense subgroup that admits an
open projection onto any factor. Such a group is not metrizable, but it is possible to refine its topology so that
one gets a metrizable group while all projections onto factors remain open.

Let $G$ be an injectively universal topological group with a countable base.
For example, we can take for $G$ the group $\Iso(U)$ of isometries
of the Urysohn space or the group $H(Q)$ of all self-homeomorphisms of the Hilbert cube. Consider the countable power 
$G^\N$ of $G$, and let $X$ be $G^\N$ equipped with a finer zero-dimensional topology with a countable base $\sB$. 
We assume that $\sB$ consists of clopen sets and is closed under complements, finite unions and hence also under finite intersections. 
In other words, $\sB$ is a Boolean algebra of clopen sets.

We consider each $x\in X=G^\N$ as an index of a certain countable metrizable group $G_x$, namely, the subgroup of $G$
generated by the elements of the sequence $x$. Every countable metrizable group is of the form $G_x$ for some $x\in X$. 
Consider the group (without topology) $C(X,G)$ of all continuous maps $f:X\to G$. We are going to construct a countable
subgroup $K\sbs C(X,G)$ and a metrizable group topology on $K$ such that for every $x\in X$ the evaluation map $f\mapsto f(x)$
from $K$ to $G$ is an open map onto $G_x$. This implies that $K$ is a projectively universal countable metrizable group.

Let $p_n\in C(X,G)$ be defined by $p_n(x)=x_n$ for $x=(x_n)$ and $n\in \N$. Let $H$ be the countable subgroup of 
$C(X,G)$ generated by all the $p_n$'s. Let $K\sbs C(X,G)$ be the countable subgroup defined as follows: if $f\in C(X,G)$, 
then $f\in K$ if and only if  
there exist a finite decomposition $X=Y_1\cup\dots\cup Y_n$, $Y_i\in\sB$, and elements $f_1,\dots,f_n\in H$ such that
$f|Y_i=f_i|Y_i$, $i=1,\dots,n$. Clearly for every $x\in X$ the evaluation map $f\mapsto f(x)$ sends $H$ and $K$ onto $G_x$. 
We are going to introduce a metrizable group topology on $K$ such that all the evaluation maps $K\to G_x$ are open. 

Let $\sN(G)$ be the filter of neighborhoods of $1_G$ in $G$ (we'll use a similar notation also for other groups). 
For $U\in \sN(G)$ let $W_U=\{f\in K: f(X)\sbs U\}$. The filter $\sF_0$ on $K$ generated by
the collection $\{W_U: U\in \sN(G)\}$ may be not invariant under inner automorphsisms; 
let $\sF$ be the smallest invariant filter containing $\sF_0$. The filter $\sF$ is generated by sets of the form $g W_U g\obr$
($U\in \sN(G)$, $g\in K$). Since $K$ is countable and $G$ is metrizable, $\sF$ has a countable base. Equip $K$ with 
the group topology for which $\sF=\sN(K)$ is the filter of neighborhoods of $1_K$. Then $K$ is metrizable.

Pick $x\in X$. The evaluation map $ev_x:K\to G_x$ defined by $ev_x(f)=f(x)$ clearly is continuous.
We must prove that it is also open. 
It suffices to check that for every $g_1,\dots, g_n\in K$ and $U\in \sN(G)$ there exists $V\in \sN(G)$ such that
$$
V\cap G_x\sbs ev_x\left(\bigcap_{i=1}^n g_iW_Ug_i\obr\right).
$$
Put $a_i=ev_x(g_i)=g_i(x)\in G$. Pick a symmetric open $V\in\sN(G)$ such that $a_i\obr V^3a_i\sbs U$ for $i=1,\dots,n$.
We check that $V$ has the required property. 

Let $b\in V\cap G_x$. Pick $f'\in H$ such that $f'(x)=b$. There exists a neighborhood $Y\in \sB$ of $x$ such that $f'(Y)\sbs V$
and $g_i(Y)\sbs Va_i$ ($i=1,\dots, n$). Let $f\in K$ be such that $f|Y=f'|Y$, $f|{X\stm Y}=1_G$.
For every $y\in Y$ and $i=1,\dots, n$ we have $g_i(y)\in Va_i$, $f(y)\in V$, hence
$$
g_i(y)\obr f(y)g_i(y)\in a_i\obr V^3 a_i\sbs U,
$$
and this is trivially true if $y\in X\stm Y$. It follows that $g_i\obr fg_i\in W_U$ and $f\in \bigcap_{i=1}^n g_iW_Ug_i\obr$.
Thus $b=f(x)\in ev_x(\bigcap_{i=1}^n g_iW_Ug_i\obr)$. We have proved that $ev_x:K\to G_x$ is open.


\begin{thebibliography}{100} 
  
\bibitem{ArhTk} A. Arhangel'skii and M. Tkachenko, {\it Topological Groups and Related Structures}, Atlantic Press / World Scientific,
Amsterdam -- Paris, 2008. 

\bibitem{BK} Howard Becker, A. S. Kechris, {\em The Descriptive Set Theory of Polish Group Actions,} Cambridge University Press, 1996.

\bibitem{Yaacov} Ita\"\i\ Ben Yaacov, {\em The linear isometry group of the Gurarij space is universal,} Proc. Amer. Math. Soc. \textbf{142} (2014), 2459--2467; arXiv:1203.4915 [math.LO].

\bibitem{Bour} N. Bourbaki, {\it Topologie g\'en\'erale: Chapitres 1 \`a 4}, Hermann, Paris, 1971.

\bibitem{GP} S. Gao, V. Pestov,
{\em On a universality property of some abelian Polish groups,}
Fund. Math. \textbf{179} (2003), 1--15. 

\bibitem{Ding} Longyun Ding,  {\it On surjectively universal Polish groups,}
Adv. Math. {\bf 231}
(2012), no. 5, 2557--2572; arXiv:1109.2283

\bibitem{Kechris} A.S. Kechris, {\em Topology and descriptive set theory,} Topology Appl. \textbf{58} (1994), 195--222.

\bibitem{SB}
R.D. Mauldin (ed.), {\em The Scottish Book,} Birkh\"auser, Boston--Basel--Stuttgart, 1981.



\bibitem{P} V.G. Pestov, {\em
Neighborhoods of identity in free topological groups}, Vestnik Moskov. Univ. Ser. I Mat. Mekh. 1985, no. 3, 8--10, 101.  (Russain); English translation: Moscow Univ. Math. Bull. \textbf{40} (1985), no. 3, 8--12.

\bibitem{RD} W. Roelcke and S. Dierolf, \textit{Uniform Structures on Topological Groups and Their Quotients,} McGraw-Hill, 1981. 

\bibitem{PSW} D. Shakhmatov, J. Pelant, and S. Watson, {\it A universal complete metric Abelian group of a
given weight,} Bolyai Soc. Math. Stud. {\bf 4} (1995), 431--439.

\bibitem{Shk} S. Shkarin, {\it On universal Abelian topological groups,} Mat. Sb. {\bf 190} (1999), 127--144 (Russian); 
English translation:
Sb. Mat. {\bf 190} (1999), 1059--1076. 

\bibitem{Usp86} V. Uspenskij, {\it A universal topological group with a countable
base,} Funktsion. analiz i ego prilozh. {\bf 20} (1986), No.~2, 86--87 (Russian);
English transl. in: Functional analysis and its appl. {\bf 20} (1986),
No.~2, 160--161.

\bibitem{Usp90} V. Uspenskij, {\it On the group of isometries of the Urysohn
universal metric space}, Comment. Math. Univ. Carolinae {\bf 31} (1990),
No.~1, 181--182.

\bibitem{Usp08} V.V. Uspenskij, {\em Unitary representability of free abelian topological groups,} Appl. Gen. Topol. \textbf{9} (2008), 197--204.
\end{thebibliography}
\end{document}